\documentclass{amsart}
\usepackage{latexsym}
\usepackage{amsfonts}
\usepackage{amsmath}
\usepackage{amssymb}
\usepackage{amsthm}
\usepackage{newlfont}

\newtheorem{pro}{Propositon}[section]
\newtheorem{theorem}{Theorem}[section]

\newtheorem{rmk}{Remark}[section]

\allowdisplaybreaks

\begin{document}

\title{Ergodic Inequalities of \\Three Population Genetic Models}

%    Remove any unused author tags.

%    author one information
\author{Youzhou Zhou}
\address{Department of Mathematics and Statistics\\
McMaster University, 1280 Main Street West,\\
Hamilton Ontario, Canada L8S 2L4}
%\curraddr{}
\email{zhouy52@math.mcmaster.ca}
%\thanks{}

\subjclass[2000]{Primary }
%    For articles to be published after 1 January 2010, you may use
%    the following version:
%\subjclass[2010]{Primary }


\date{\today}

\dedicatory{}

\begin{abstract}
In this article, three models are considered, they are the infinitely-many-neutral-alleles model \cite{MR615945}, infinite dimensional diffusion associated with two-parameter Poisson-Dirichlet distribution \cite{MR2596654} and the infinitely-many-alleles model with symmetric dominance \cite{MR1626158}. The new representations of the transition transition densities are obtained for the first two models. Lastly, the ergodic inequalities of these three models are provided.
\end{abstract}

\maketitle

\section{Introduction}

Fleming-Viot process is the most general model in population genetics, it can include various evolutionary forces in a single model, such as mutations, selections.  Let $E$ be the type space, and $\mathcal{P}(E)$ be the set of probability measures on $E$, then Fleming-Viot process $Z_{t}$ is a $\mathcal{P}(E)$-valued diffusion process, with generator,
\begin{align*}
AF_{f}(\mu)=&\sum_{1\leq i<j\leq m}(\langle\Phi_{ij}^{(m)}f,\mu^{m-1}\rangle-\langle f,\mu^{m}\rangle)+\langle B^{(m)}f,\mu^{m}\rangle\\
+&2\bar{\sigma}\sum_{i=1}^{m}(\langle K_{i}^{(m)}f,\mu^{m+2}\rangle-\langle f,\mu^{m}\rangle)+\bar{\sigma}m\langle f,\mu^{m}\rangle,
\end{align*}
where $\mu\in\mathcal{P}(E)$ and $f\in\mathcal{B}(E^{m})$. $\Phi_{ij}^{(m)}f$ is called sampling operator, which replace the $jth$ variable of $f$ by the $ith$ variable. $Bf$ is called mutation operator, which generates a Feller semigroup $\{T_{t},t\geq0\}$ defined by transition probability $P(t,x,dy)$, and $B^{(m)}$ is the generator of semigroup
$$
T_{m}(t)f=\int_{E}\cdots\int_{E}f(y_{1},\cdots,y_{m})P(t,x_{1},dy_{1})\cdots P(t,x_{m},dy_{m}).
$$
$K_{i}^{(m)}$ is called selection operator defined
$$
K_{i}^{(m)}f=\frac{\bar{\sigma}+\sigma(x_{i},x_{m+1})-\sigma(x_{m+1},x_{m+2})}{2\bar{\sigma}}f(x_{1},\cdots, x_{m}).
$$
$\sigma(x,y)$ is a symmetric function and called relative fitness of genotype $\{x,y\}$. $\bar{\sigma}$ is defined to be $\sup_{x,y,z}|\sigma(x,y)-\sigma(y,z)|$. For more comprehensive introduction to Fleming-Viot process, please refer to survey paper \cite{MR1205982}.

 If the mutation operator $B$ of Fleming-Viot process $Z_{t}$ is of the form
 $$
 Bf(x)=\frac{\theta}{2}\int_{E}(f(y)-f(x))\nu_{0}(dx),\quad \theta>0, \nu_{0}\in\mathcal{P}(E),
 $$
 then $\forall t>0$, $Z_{t}$ is almost surely of purely atomic measure. Denote the totality of purely atomic measures by $\mathcal{P}_{a}$. For $\mu\in\mathcal{P}_{a}$, if we consider the decreasing arrangement of the atomic mass of $\mu$, then we will end up with $(x_{1},x_{2},\cdots)$, which consists of a set 
$$
\bar{\triangledown}_{\infty}=\left\{(x_{1},x_{2},\cdots)\bigg|x_{1}\geq x_{2}\geq\cdots\geq0, \sum_{i=1}^{\infty}x_{i}\leq1\right\}.
$$
We can define an atomic mapping $\rho: \mathcal{P}(E)\rightarrow\bar{\triangledown}_{\infty}$ by mapping $\mu$ to its decreasingly ordered atomic vector $(x_{1},x_{2},\cdots)$. Therefore, $\rho(Z_{t})$ is $\bar{\triangledown}_{\infty}$-valued process. We call Fleming-Viot process labelled model and its atomic process $\rho(Z_{t})$ unlabelled model. 

  If there are only random sampling and mutations involved, then $\rho(Z_{t})$ is the infinitely-many-neutral-alleles model \cite{MR615945}, denoted by $X_{t}$, the generator of which is
 $$
G=\frac{1}{2}\sum_{i,j=1}^{\infty}x_{i}(\delta_{i,j}-x_{j})\frac{\partial^{2}}{\partial x_{i}\partial x_{j}}-\frac{\theta}{2}\sum_{i=1}^{\infty}x_{i}\frac{\partial}{\partial x_{i}}, x\in \bar{\triangledown}_{\infty}.
$$
 
If we include selection as well, then the unlabelled model is usually non-Markovian. But if we consider selection of symmetric dominance introduced in \cite{MR0504021}, then the unlabelled model is a Markov process. We denote this unlabelled model by $X_{t}^{\sigma}$ and call it infinitely-many-alleles diffusion with symmetric dominance. Please refer to \cite{MR1626158} for details. The generator of $X_{t}^{\sigma}$ is
$$
G_{\sigma}=G+\sigma\sum_{i=1}^{\infty}x_{i}(x_{i}-\varphi_{2}(x))\frac{\partial}{\partial x_{i}}, x\in\bar{\triangledown}_{\infty}.
$$ 
Both $X_{t}$ and $X_{t}^{\sigma}$ are reversible diffusions and have unique stationary distributions. The stationary distribution of $X_{t}$ is Poisson-Dirichlet distribution PD($\theta$), and the stationary distribution of $X_{t}^{\sigma}$ is 
$$
\pi_{\sigma}(dx)=C_{\sigma}\exp\{\sigma\varphi_{2}(x)\}\text{PD}(\theta)(dx),
$$
where $C_{\sigma}$ is a normalized constant and $\varphi_{2}(x)=\sum_{i=1}^{\infty}x_{i}^{2}$ is the homozygosity in population genetics.

Moreover, there is a two-parameter generalization of PD($\theta$), we call it two-parameter Poisson-Dirichlet distribution (refer to \cite{MR2663265}) PD($\theta,\alpha$), $\theta+\alpha>0, 0<\alpha<1$. Correspondingly, there is a two-parameter generalization \cite{MR2596654},\cite{MR2678897} of $X_{t}$. We denote it by $X_{t}^{\theta,\alpha}$ and call it two-parameter infinite dimensional diffusion whose stationary distribution is PD($\theta,\alpha$). The generator of $X_{t}^{\theta,\alpha}$ is
$$
G^{\theta,\alpha}=\frac{1}{2}\sum_{i,j=1}^{\infty}x_{i}(\delta_{i,j}-x_{j})\frac{\partial^{2}}{\partial x_{i}\partial x_{j}}-\frac{1}{2}\sum_{i=1}^{\infty}(\theta x_{i}+\alpha)\frac{\partial}{\partial x_{i}}, x\in \bar{\triangledown}_{\infty}.
$$
However, $X_{t}^{\theta,\alpha}$ has no biological interpretation at all, and whether its corresponding labelled model exists is still open. 

In \cite{MR1235429}, the transition probability of  neutral Fleming-Viot process $Z_{t}$ is obtained. In \cite{MR1174426}, the transition density function of unlabelled neutral process $X_{t}$ is also obtained. Therefore its explicit transition probability is available as well. We can actually get the transition probability of $X_{t}$ through that of $Z_{t}$, as is done in \cite{MR2663265}.  In \cite{MR2737405}, the transition density function of $X_{t}^{\theta,\alpha}$ is obtained as well. In this paper, we reorganize the transition density functions of $X_{t}$ and $X_{t}^{\theta,\alpha}$, and new representations of the transition density functions of  $X_{t}$ and $X_{t}^{\theta,\alpha}$ are obtained respectively. The associated transition probabilities resembles the structure of transition probabilities in neutral Fleming-Viot process. This can actually shed some light to the construction of corresponding labelled model of $X_{t}^{\theta,\alpha}$.

Furthermore, the ergodic inequalities of $Z_{t}$ and $X_{t}$ are both available, but similar ergodic inequalities of $X_{t}^{\theta,\alpha}$ and $X_{t}^{\sigma}$ are still missing. In this article, we have obtained the ergodic inequality of $X_{t}^{\theta,\alpha}$ and $X_{t}^{\sigma}$. Especially, for $\theta>0$, $X_{t}^{\theta,\alpha}$ and $X_{t}$ share the exactly the same ergodic inequality. Lastly, the ergodic inequality of $X_{t}^{\sigma}$ is stronger than the ergodic theorem stated in \cite{MR1626158}. 

 This paper is organized as follows: in section 2, we will talk about the transition density functions of $X_{t}^{\theta,\alpha}$ and $X_{t}^{\sigma}$. In section 3, the ergodic inequalities of them will also be discussed.

\section{ transition density functions of $X_{t}^{\theta,\alpha}$ and $X_{t}$}

In \cite{MR615945} and \cite{MR2737405}, the explicit transition densities of $X_{t}^{\theta,\alpha}$ and $X_{t}$ are obtained respectively through eigen expansion. By making use of these known transition densities, we get new representations. 

\begin{theorem}\label{neutral_density}
$X_{t}$ has the following transition density
$$
p(t,x,y)=d_{0}^{\theta}(t)+d_{1}^{\theta}(t)+\sum_{n=2}^{\infty}d_{n}^{\theta}(t)p_{n}(x,y),
$$
where
\begin{align*}
&d_{n}^{\theta}(t)=\sum_{n=2}^{\infty}\frac{2m+\theta-1}{m!}(-1)^{m-n}\binom{m}{n}(n+\theta)_{(m-1)}e^{-\lambda_{m}t},n\geq1.\\
&d_{0}^{\theta}(t)=1-\sum_{m=1}^{\infty}\frac{2m+\theta-1}{m!}(-1)^{m-1}\theta_{(m-1)}e^{-\lambda_{m}t}.\\
&p_{n}(x,y)=\sum_{|\eta|=n}\frac{p_{\eta}(x)p_{\eta}(y)}{\int p_{\eta}d\mbox{PD}(\theta)}, \eta=(\eta_{1},\cdots,\eta_{l})\mbox{ is a partition of } n.
\end{align*}
$p_{\eta}(x)$ is the continuous extension of 
$$
\frac{n!}{\eta_{1}!\cdots\eta_{l}!a_{1}!\cdots a_{n}!}\sum_{i_{1},\cdots,i_{l}\neq}x_{i_{1}}^{\eta_{1}}\cdots x_{i_{l}}^{\eta_{l}}.
$$

Define $\nu_{n}(x,dy)=p_{n}(x,y)PD(\theta)(dy)$, then the transition probability of $X_{t}$ is
$$
P(t,x,dy)=\Big(d_{0}^{\theta}(t)+d_{1}^{\theta}(t)\Big)PD(\theta)(dy)+\sum_{n=2}^{\infty}d_{n}^{\theta}(t)\nu_{n}(x,dy).
$$
\end{theorem}
\begin{proof}
 The transition density of $X_{t}$ is 
$$
p(t,x,y)=1+\sum_{m=2}^{\infty}e^{-\lambda_{m}t}Q_{m}(x,y), \lambda_{m}=\frac{m(m+\theta-1)}{2},
$$
where
\begin{align*}
&Q_{m}(x,y)=\frac{2m+\theta-1}{m!}\sum_{n=0}^{m}(-1)^{m-n}\binom{m}{n}(n+\theta)_{(m-1)}p_{n}(x,y).
\end{align*}
 Then by Fubini's theorem,
  we can rearrange $p(t,x,y)$ by switching the order of summation. 
  \begin{align*}
 p(t,x,y)=&1+\sum_{m=2}^{\infty}e^{-\lambda_{m}t}\Big(\sum_{n=2}^{m}\frac{2m+\theta-1}{m!}(-1)^{m-n}\binom{m}{n}(n+\theta)_{(m-1)}p_{n}(x,y)\\
 &+\frac{2m+\theta-1}{m!}(-1)^{m-1}(\theta+1)_{(m-1)}mp_{1}(x,y)+\frac{2m+\theta-1}{m!}(-1)^{m}\theta_{(m-1)}p_{0}(x,y)\Big)\\
 \intertext{(\mbox{ for $ p_{1}(x,y),p_{0}(x,y)=1,$  we have })}
 &=1+\sum_{m=2}^{\infty}e^{-\lambda_{m}t} \sum_{n=2}^{m}\frac{2m+\theta-1}{m!}(-1)^{m-n}\binom{m}{n}(n+\theta)_{(m-1)}p_{n}(x,y)\\
 &+\sum_{m=2}^{\infty}e^{-\lambda_{m}t}\Big(\frac{2m+\theta-1}{m!}(-1)^{m-1}(\theta+1)_{(m-1)}m+\frac{2m+\theta-1}{m!}(-1)^{m}\theta_{(m-1)}\Big) \\
 =&1+\sum_{m=2}^{\infty}e^{-\lambda_{m}t} \sum_{n=2}^{m}\frac{2m+\theta-1}{m!}(-1)^{m-n}\binom{m}{n}(n+\theta)_{(m-1)}p_{n}(x,y)\\ 
 &+\sum_{m=2}^{\infty}e^{-\lambda_{m}t}\frac{2m+\theta-1}{m!}(-1)^{m-1} [m(\theta+1)_{(m-1)}-\theta_{(m-1)}]\\
 \intertext{ since when $ m=1,  m(\theta+1)_{(m-1)}-\theta_{(m-1)}=0,$ then we have }
 =&1-\sum_{m=1}^{\infty}e^{-\lambda_{m}t}\frac{2m+\theta-1}{m!}(-1)^{m-1}\theta_{(m-1)}\\
 &+\sum_{m=1}^{\infty}e^{-\lambda_{m}t}\frac{2m+\theta-1}{m!}(-1)^{m-1}m(\theta+1)_{(m-1)}\\
 &+\sum_{m=2}^{\infty}e^{-\lambda_{m}t} \sum_{n=2}^{m}\frac{2m+\theta-1}{m!}(-1)^{m-n}\binom{m}{n}(n+\theta)_{(m-1)}p_{n}(x,y)\\
 =&d_{0}^{\theta}(t)+d_{1}^{\theta}(t)+ \sum_{m=2}^{\infty}e^{-\lambda_{m}t} \sum_{n=2}^{m}\frac{2m+\theta-1}{m!}(-1)^{m-n}\binom{m}{n}(n+\theta)_{(m-1)}p_{n}(x,y).
 \end{align*}
Switching the order of summation, we have
  $$
 p(t,x,y)= d_{0}^{\theta}(t)+d_{1}^{\theta}(t)+\sum_{n=2}^{+\infty}d_{n}^{\theta}(t)p_{n}(x,y).
 $$ 
 \end{proof}
\begin{theorem}
$X_{t}^{\theta,\alpha}$ has the following transition density
$$
p^{\theta,\alpha}(t,x,y)=d_{0}^{\theta}(t)+d_{1}^{\theta}(t)+\sum_{n=2}^{\infty}d_{n}^{\theta}(t)p_{n}^{\theta,\alpha}(x,y),
$$
where $d_{n}^{\theta}(t), n\geq0,$ are defined in theorem \ref{neutral_density} and
$$
p_{n}^{\theta,\alpha}(x,y)=\sum_{|\eta|=n}\frac{p_{\eta}(x)p_{\eta}(y)}{\int p_{\eta}dPD(\theta,\alpha)}, \eta=(\eta_{1},\cdots,\eta_{l})\mbox{ is a partition of } n.
$$

Define $\nu_{n}^{\theta,\alpha}(x,dy)=p_{n}(x,y)PD(\theta,\alpha)(dy)$, then the transition probability of $X_{t}$ is
\begin{align}
P^{\theta,\alpha}(t,x,dy)=\Big(d_{0}^{\theta}(t)+d_{1}^{\theta}(t)\Big)PD(\theta,\alpha)(dy)+\sum_{n=2}^{\infty}d_{n}^{\theta}(t)\nu_{n}^{\theta,\alpha}(x,dy).\label{two_parameter_representation}
\end{align}
\end{theorem}
\begin{proof}
The proof of this theorem is quite similar to theorem \ref{neutral_density}, thereby omitted.
\end{proof}
\begin{rmk}
Since $X_{t}$ has an entrance boundary $\bar{\triangledown}_{\infty}-\triangledown_{\infty}$, i.e. $X_{t}$ will immediately moves into $\triangledown_{\infty}$ and never exits regardless of its starting point. Similarly, we can show the similar result for $X_{t}^{\theta,\alpha}$ informed by S.N. Ethier.
\end{rmk}
For both $X_{t}$ and $X_{t}^{\theta,\alpha}$, the structures of transition probability are so similar. They even share the coefficients $d_{n}^{\theta}(t),n\geq0,$ which is the entrance of the ancestral process discussed by Simon Tavar\'e in \cite{MR770050}. But Tavar\'e constructed this process only when $\theta>0$. In fact, if we collapse the state $0$ and $1$, and relabel it as $1$, this is essentially Kingman coalescence with mutation. We can generalize this structure to the case where $\theta>-1.$
\begin{pro}\label{two_tavare}
For $\theta>-1$, we have 
$$
 e^{-\lambda_{n}t}\leq \sum_{k=n}^{\infty}d_{k}^{\theta}(t)\leq \frac{(n+\theta)_{(n)}}{n_{[n]}}e^{-\lambda_{n}t}. 
 $$
 In particular, when $n=2$, we know 
 \begin{equation}
 \sum_{k=2}^{\infty}d_{k}^{\theta}(t)\leq \frac{(2+\theta)(3+\theta)}{2}e^{-(\theta+1)t}.\label{tail_estimation}
 \end{equation}
 \end{pro}
 \begin{proof}
 Consider a pure-death Markov chain $B_{t}$ in $\{1,2,\cdots,m\}$ with Q matrix,
 $$
 Q=\begin{pmatrix}
 0 &0 & 0& \cdots &0 &0\\
 \lambda_{2}& -\lambda_{2} &0 &\cdots &0 &0\\
 0 & \lambda_{3} & -\lambda_{3}&\cdots&0 &0\\
 \vdots&\vdots &\vdots&\cdots&\vdots&\vdots\\
 0& 0&0&\cdots &\lambda_{m} &-\lambda_{m}
 \end{pmatrix}
 $$
 where $\lambda_{k}=\frac{k(k+\theta-1)}{2}, k\geq2.$  Running the similar arguments in Theorem 4.3 in \cite{MR2663265}, we will be able to find all the left eigenvectors and right eigenvectors of $Q$. Denote the matrix consisting of left eigenvectors by $U=(u_{ij})$ and the matrix consisting of right eigenvectors by $V=(v_{ij})$, where
 $$
 u_{ij}=\begin{cases}
 \delta_{1j} & i=1\\
 0 & j>i>1\\
 (-1)^{i-j}\binom{i}{j}\frac{(j+\theta)_{(i-1)}}{(i+\theta)_{(i-1)}}, & j\leq i, i>1,
 \end{cases}
 $$
 and
  $$
 v_{ij}=\begin{cases}
 1 & j=1\\
 0 & j>i\\
 \binom{i}{j}\frac{(j+\theta)_{(j)}}{(i+\theta)_{(j)}}, & 1< j\leq i.
 \end{cases}
 $$
 Note that the row vectors of $U$ are left eigenvectors of $Q$ and the column vectors of $V$ are the right eigenvectors of $Q$. Similarly, we can also show that $UV=I$ and $Q$ is diagonlized as $V\Lambda U$, where $\Lambda=diag\{0,-\lambda_{2},\cdots,-\lambda_{m}\}$.  Therefore, the transition matrix $P_{t}$ is 
 $$
 P_{t}=e^{tQ}=Ve^{\Lambda t}U.
 $$
 By direct computation, we know, for $2\leq n\leq m$,
 $$
 P_{mn}(t)=\sum_{k=n}^{m}(-1)^{k-n}\binom{m}{k}\binom{k}{n}\frac{(\theta+k)_{(k)}}{(\theta+m)_{(k)}}\frac{(\theta+n)_{(k-1)}}{(\theta+k)_{(k-1)}}e^{-\lambda_{k}t}.
 $$
 Letting $m\rightarrow+\infty$, we have $d_{n}^{\theta}(t)=\lim_{m\rightarrow\infty}P_{mn}(t).$ 
 
 The remaining arguments are essentially due to Tavar\'e.
 
  By the martingale argument in chapter 6 of \cite{MR611513}, we know 
 $$
 Z_{n}(t)=\frac{e^{\lambda_{n}t}(B_{t})_{[n]}}{(B_{t}+\theta)_{(n)}},
 $$
 because $e^{-\lambda_{n}t}$ is one eigenvalue of $P_{t}$ and  
 $$
 (0,0,\cdots, 0,\frac{n_{[n]}}{(n+\theta)_{(n)}}, \cdots, \frac{k_{[n]}}{(k+\theta)_{(n)}}, \cdots,\frac{m_{[n]}}{(m+\theta)_{(n)}})^{T}
 $$ is the corresponding eigenvector. So 
 $$
 EZ_{n}(t)=Z_{n}(0)=\frac{m_{[n]}}{(m+\theta)_{(n)}}.
 $$
 Since, for $n\leq k\leq m$,
 $$
 \frac{n_{[n]}}{(n+\theta)_{(n)}}\leq\frac{k_{[n]}}{(k+\theta)_{(n)}}\leq \frac{m_{[n]}}{(m+\theta)_{(n)}},
 $$
 and 
 $$
 \frac{e^{-\lambda_{n}t}m_{[n]}}{(m+\theta)_{(n)}}=e^{-\lambda_{n}t}EZ_{n}(t)=\sum^{m}_{k=n} \frac{k_{[n]}}{(k+\theta)_{(n)}}P_{mk}(t),
  $$
  we have
  $$
  \frac{n_{[n]}}{(n+\theta)_{(n)}}P(B_{t}\geq n|B_{0}=m)\leq \frac{e^{-\lambda_{n}t}m_{[n]}}{(m+\theta)_{(n)}}\leq \frac{m_{[n]}}{(m+\theta)_{(n)}} P(B_{t}\geq n|B_{0}=m).
 $$
 Thus, we have
 $$
 e^{-\lambda_{n}t}\leq P(B_{t}\geq n|B_{0}=m)\leq \frac{(n+\theta)_{(n)}}{n_{[n]}}e^{-\lambda_{n}t}.
 $$
 Letting $m\rightarrow\infty$, we have
 $$
 e^{-\lambda_{n}t}\leq \sum_{k=n}^{\infty}d_{k}^{\theta}(t)\leq \frac{(n+\theta)_{(n)}}{n_{[n]}}e^{-\lambda_{n}t}. 
 $$
 \end{proof}

\section{Ergodic Inequalities}

By making use of the transition probability (\ref{two_parameter_representation}) and the tail probability estimation (\ref{tail_estimation}), we can easily get the following ergodic inequality of $X_{t}^{\theta,\alpha}.$
\begin{theorem}\label{T3}
For $X_{t}^{\theta,\alpha}$, we have the ergodic inequality
$$
\sup_{x\in\bar{\triangledown}_{\infty}}||P^{\theta,\alpha}(t,x,\cdot)-PD(\theta,\alpha)(\cdot)||_{\mbox{var}}\leq \frac{(2+\theta)(3+\theta)}{2}\exp\{-(\theta+1)t\}, t\geq0.
$$
\end{theorem}
\begin{proof} 

\begin{align*}
&||P^{\theta,\alpha}(t,x,\cdot)-PD(\theta,\alpha)(\cdot)||_{\mbox{var}}
\leq \sup_{A\in\mathcal{B}}|P^{\theta,\alpha}(t,x,A)-PD(\theta,\alpha)(A)|\\
&=|(d_{0}^{\theta}(t)+d_{1}^{\theta}(t))PD(\theta,\alpha)(A)+\sum_{n=2}^{\infty}d_{n}^{\theta}(t)\nu_{n}^{\theta,\alpha}(A)-PD(\theta,\alpha)(A)|\\
&=|\sum_{n=2}^{\infty}d_{n}^{\theta}(t)(\nu_{n}^{\theta,\alpha}(A)-PD(\theta,\alpha)(A))|\\
&\leq \sum_{n=2}^{\infty}d_{n}^{\theta}(t)|\nu_{n}^{\theta,\alpha}(A)-PD(\theta,\alpha)(A)|\\
&\leq\sum_{n=2}^{\infty}d_{n}^{\theta}(t)\leq \frac{(\theta+2)(\theta+3)}{2}e^{-(\theta+1)t}.
\end{align*}
\end{proof}

\begin{pro}\label{ultra}
The transition densities $p_{\sigma}(t,x,y)$ of $X_{t}^{\sigma}$ is also ultra-bounded, i,e.
$$
p_{\sigma}(t,x,y)\leq \frac{1}{C_{\sigma}}e^{|\sigma|(1+\theta)t+\sigma^{2}+3|\sigma|}ct^{\frac{c\log t}{t}}.
$$
\end{pro}
\begin{proof}
This estimation can be easily obtained from (4.17) in \cite{MR1626158} and theorem 3.3 in \cite{MR2737405}.
\end{proof}
Since the one-parameter selective model is absolutely continuous with respect to one-parameter neutral model, $\bar{\triangledown}_{\infty}-\triangledown_{\infty}$ should also serve as an entrance boundary. Hence we can change the value of the density function $p_{\sigma}(t,x,y)$ when $x$ or $y$ is in $\bar{\triangledown}_{\infty}-\triangledown_{\infty}$. Therefore, $p_{\sigma}(t,x,y)$ can be chosen to be the continuous extension of $p_{\sigma}(t,x,y)\big|_{\triangledown_{\infty}\times\triangledown_{\infty}}$. Moreover, $p_{\sigma}(t,x,y)$ is symmetric for $X_{t}^{\sigma}$ is reversible. By proposition \ref{ultra}, the Poincar\'{e} inequality of $X_{t}^{\sigma}$ also holds. It, therefore, guarantees the $L_{2}$-exponential convergence. By running the argument in theorem 8.8 in \cite{MR2105651}, we can also get the following ergodic inequality.
\begin{theorem}
For $X_{t}^{\sigma}$ , $\exists K(\theta,\sigma),$ such that
$$
\sup_{x\in\bar{\triangledown}_{\infty}}\|P^{\sigma}(t,x,\cdot)-\pi_{\sigma}(\cdot)\|_{\text{var}}\leq K(\theta,\sigma)\exp\{-(\text{{\rm gap}}(G_{\sigma}))t\}, t\geq0.
$$
\end{theorem}
\begin{proof} 
We are going to run the argument in theorem 8.8 in \cite{MR2105651}.  Since 
$$
P^{\sigma}(t,x,\cdot)=\int_{\bar{\triangledown}_{\infty}} P^{\sigma}(t-s,z,\cdot)P^{\sigma}(s,x,dz),
$$
and define $\mu^{x}(\cdot)=P_{x}(X_{s}^{\sigma}\in \cdot)$, we have 
$$
P^{\sigma}(t,x,\cdot)=\mu^{x}P_{t-s}^{\sigma}(\cdot).
$$
Therefore, 
\begin{align*}
\|P^{\sigma}(t,x,\cdot)-\pi(\cdot)\|_{var}=\|\mu^{x}P_{t-s}^{\sigma}(\cdot)-\pi(\cdot)\|_{ var}.
\end{align*}
 By part (1) in theorem 8.8 in \cite{MR2105651}, we have, $\forall t\geq s$,
 \begin{align*}
\|P^{\sigma}(t,x,\cdot)-\pi(\cdot)\|_{var}\leq&\left\|\frac{d\mu^{x}}{d\pi_{\sigma}}-1\right\|_{2}e^{-(t-s)\text{gap}(G_{\sigma})} \\
=&\sqrt{\int p(s,x,y)^{2}\pi_{\sigma}(dy)-1}e^{-(t-s)\text{gap}(G_{\sigma})}
\end{align*}
Therefore, for $t\geq\frac{1}{2}$, we have
$$
\|P^{\sigma}(t,x,\cdot)-\pi(\cdot)\|_{var}\leq\sqrt{\int p(s,x,y)^{2}\pi_{\sigma}(dy)-1}e^{s\cdot\text{gap}(G_{\sigma})}\exp\{- \text{gap}(G_{\sigma})t\}.
$$
 If we choose $s=\frac{1}{2}$, then by Proposition \ref{ultra}, the constant 
 \begin{align*}
  K^{'}(\theta,\sigma)=\sqrt{2ce^{|\sigma|(1+\theta)+\sigma^{2}+3|\sigma|}}e^{\frac{1}{2}\text{gap}(G_{\sigma})}
 \geq\sqrt{\int p(s,x,y)^{2}\pi_{\sigma}(dy)-1}e^{\frac{1}{2}\text{gap}(G_{\sigma})}.
 \end{align*} 
 Then we have
 $$
 \sup_{x\in\bar{\triangledown}_{\infty}}\Big\|P^{\sigma}(t,x,\cdot)-\pi_{\sigma}(\cdot)\Big\|_{var} \leq K^{'}(\theta,\sigma)\exp\{-\text{gap}(G_{\sigma})t\}, \forall t\geq\frac{1}{2}.
 $$ 
 Moreover,
 $$
 \sup_{x\in\bar{\triangledown}_{\infty}}\Big\|P^{\sigma}(t,x,\cdot)-\pi_{\sigma}(\cdot)\Big\|_{var}\leq 1,\quad \forall t\geq0.
 $$
 Thus $\forall t\in[0,\frac{1}{2}],$ if we choose $K^{''}(\theta,\sigma)$ such that
 $$
 K^{''}(\theta,\sigma)e^{-\text{gap}(G_{\sigma})/2}\geq1,
 $$
then $\forall t\in[0,\frac{1}{2}],$
 \begin{align*}
 \sup_{x\in\bar{\triangledown}_{\infty}}\Big\|P^{\sigma}(t,x,\cdot)-\pi_{\sigma}(\cdot)\Big\|_{var}\leq&1
\leq K^{''}(\theta,\sigma)e^{-\text{gap}(G_{\sigma})/2}
\leq K^{''}(\theta,\sigma)\exp\{-\text{gap}(G_{\sigma})t\}.
 \end{align*}
Therefore, choosing $K(\theta,\sigma)=\max\{K^{'}(\theta,\sigma),K^{''}(\theta,\sigma)\}$, we have
$$
\sup_{x\in\bar{\triangledown}_{\infty}}\Big\|P^{\sigma}(t,x,\cdot)-\pi_{\sigma}(\cdot)\Big\|_{var}\leq K(\theta,\sigma) \exp\{- \text{gap}(G_{\sigma})t\}.
$$
\end{proof}

\section{Acknowledgement}

 This paper is part of my PhD work. I am so grateful for the support and mentor of my PhD advisor, Shui Feng. Moreover, I want to thank Stewart N. Ethier for his suggestion in the proof of proposition \ref{two_tavare}.

\bibliography{reference}{}
\bibliographystyle{amsplain}
\end{document}